\documentclass[10pt, a4paper]{amsart}

\usepackage{etex}
\usepackage[applemac]{inputenc}
\usepackage{textcomp} 
\usepackage{graphicx} 
\usepackage{flafter}
\usepackage{verbatim}
\usepackage[all]{xy}
\usepackage{epsfig,enumerate}
\usepackage{algorithmic}
\usepackage{amsmath}
\usepackage{amsthm}
\usepackage{amssymb}
\usepackage{mathrsfs}
\usepackage{amsfonts}
\usepackage{amscd}
\usepackage{euscript}
\usepackage{amssymb}
\usepackage{mathrsfs}
\usepackage{marvosym}
\usepackage[nice]{nicefrac}
\usepackage{bm}
\usepackage{parallel}
\usepackage[pdftex,bookmarks,breaklinks]{hyperref}
\numberwithin{equation}{section}
\usepackage{memhfixc}
\usepackage{pdfsync}
\usepackage{microtype}
\usepackage{booktabs}
\usepackage[a4paper,top=3cm,bottom=3cm,left=3.35cm,right=3.35cm, bindingoffset=5mm]{geometry}
\usepackage{setspace}
\usepackage{fancyhdr}
\usepackage{indentfirst}
\usepackage{footnote}
\usepackage{booktabs}
\usepackage{listings}
\usepackage{enumitem}
\usepackage{dsfont}
\usepackage{xcolor}
\usepackage{scrextend}
\usepackage{chngcntr}
\usepackage[OT2,T1]{fontenc}
\usepackage{tikz-cd}
\newcommand{\set}[1]{\left\{ #1 \right\}}

\newcommand{\ra}{\rightarrow}

\newcommand{\isom}{\xrightarrow{\sim}}
\newcommand{\map}[5]{#1: \xymatrix@R=0in{#2 \ar[r] & #3 \\ #4 \ar@{|->}[r] & #5}}
\renewcommand{\varphi}{\phi}

\newcommand{\Hom}{\mathrm{Hom}}

\newcommand{\Spec}{\mathrm{Spec} \,}

\newcommand{\alg}[1]{\overline{#1}}
\newcommand{\sep}[1]{{#1}^{\rm{s}}}
\newcommand{\Gal}{\mathrm{Gal}}

\newcommand{\ad}{\mathrm{ad}}

\newcommand{\Z}{\mathbb{Z}}
\newcommand{\Q}{\mathbb{Q}}
\newcommand{\R}{\mathbb{R}}

\newcommand{\GL}{\mathrm{GL}}

\newcommand{\gl}{\mathfrak{gl}}

\DeclareSymbolFont{cyrletters}{OT2}{wncyr}{m}{n}
\DeclareMathSymbol{\Sha}{\mathalpha}{cyrletters}{"58}

\theoremstyle{plain}
\newtheorem{theorem}{Theorem}
\newtheorem*{theorem*}{Theorem}

\newtheorem{corollary}[theorem]{Corollary}
\newtheorem{lemma}[theorem]{Lemma}

\theoremstyle{definition}

\newtheorem{remark}[theorem]{Remark}

\numberwithin{equation}{section}
\numberwithin{theorem}{section}

\begin{document}

\title{Integral points on elliptic curves and Modularity}
\author{Federico Amadio Guidi}
\date{}
\address{Mathematical Institute \\ University of Oxford \\ Oxford, UK}
\email{federico.amadio@maths.ox.ac.uk}

\maketitle

\begin{abstract}
In this paper we prove the finiteness of the set of $S$-integral points of a punctured rational elliptic curve without complex multiplication using the Chabauty-Kim method. This extends previous results of Kim \cite{kim10} in the complex multiplication case. The key input of our approach is the use of modularity techniques to prove the vanishing of certain Selmer groups involved in the Chabauty-Kim method.
\end{abstract}

\section*{Introduction}

Let $E$ be an elliptic curve with no complex multiplication over $\Q$, and let $X$ be the hyperbolic genus $1$ curve over $\Q$ obtained by removing the origin from $E$. Let $\mathcal{E}$ be a Weierstrass minimal model for $E$, and let $\mathcal{X}$ be the integral model of $X$ obtained as the complement in $\mathcal{E}$ of the Zariski closure of the origin. Let $S$ be a finite set of rational primes including $\infty$ and the primes of bad reduction for $\mathcal{E}$. We consider the set $\mathcal{X} (\Z_S)$ of $S$-integral points of $X$.

A classical theorem of Siegel, see for instance \cite{lang83}, states that $\mathcal{X} (\Z_S)$ is finite. In this paper, we reprove this finiteness result using the Chabauty-Kim method.

Given a prime $p$ of good reduction for $\mathcal{E}$, this method produces a nested sequence
\[ \mathcal{X}(\Z_p) \supset \mathcal{X} (\Z_p)_1 \supset \mathcal{X} (\Z_p)_2 \supset \cdots \supset \mathcal{X} (\Z_p)_n \supset \cdots \supset \mathcal{X} (\Z_S) \]
of sets of $p$-adic points, each containing $\mathcal{X} (\Z_S)$. A precise description of these sets is given in \S \ref{Chabauty-Kim}. The main result of this paper is the following.

\begin{theorem*} [{see Corollary \ref{main_corollary}}]
The sets $\mathcal{X} (\Z_p)_n$ are finite for $n$ sufficiently large.
\end{theorem*}

This clearly implies the finiteness of $\mathcal{X} (\Z_S)$. Moreover, note that combining this with the results of \cite{kim10} in the complex multiplication case we get finiteness of the set of $S$-integral points of any rational elliptic curve minus the origin via the Chabauty-Kim method.

The key advantage of proving a finiteness result of this kind using the Chabauty-Kim method is that the finiteness of $\mathcal{X} (\Z_p)_n$, and hence of $\mathcal{X} (\Z_S)$, is obtained by showing that $\mathcal{X} (\Z_p)_n$ is contained inside the zero locus of a $p$-adic analytic function which can be described in terms of $p$-adic iterated integrals, see \cite{kim09}. This zero locus turns out to be finite, and hence an effective bound on its cardinality naturally yields an effective bound on the cardinality of $\mathcal{X} (\Z_S)$.

The structure of this paper is the following. In \S \ref{Chabauty-Kim} we give a quick introduction to the Chabauty-Kim method, and state our main finiteness theorem. In \S \ref{Block-Kato_Selmer} we prove the key vanishing result for Bloch-Kato Selmer groups using modularity techniques. From this, using some Iwasawa theory, we deduce in \S \ref{Selmer} a vanishing result for Selmer groups, which will allow us to prove the main theorem in \S \ref{proof_main_theorem}.

\subsection*{Acknowledgements}

The author would like to thank Minhyong Kim and Andrew Wiles for their very helpful comments, suggestions, and feedbacks. The author would also like to thank Laura Capuano for checking a preliminary version of this manuscript.

\subsection*{Notation}

For a field $M$, we fix an algebraic closure $\alg{M}$ of $M$, a separable closure $\sep{M}$ of $M$ inside $\alg{M}$, and we let $G_M = \Gal (\sep{M} / M)$ be the absolute Galois group of $M$. If $M$ is an algebraic extension of $\Q$, and $T$ is a finite set of places of $M$, we let $M_T$ be the maximal extension of $M$ inside $\alg{M}$ which is unramified outside $T$, and we let $G_{M, T} = \Gal (M_T / M)$. Given a prime $p$, for every integer $n \geq 1$ we denote by $\zeta_{p^n}$ a primitive $p^n$-th root of unity in $\sep{M}$, we let $M(\zeta_{p^\infty}) = \cup_{n \geq 1} M(\zeta_{p^n})$, and we denote by $\epsilon$ the $p$-adic cyclotomic character of $G_M$. If $M'$ is a separable quadratic extension of $M$, we let $\delta_{M'/M}$ denote the nontrivial character of $\Gal (M'/M)$. Moreover, for an algebraic variety $X$ over $M$, and an extension $M'$ of $M$, we let $X_{M'} = X \times_{\Spec M} \Spec M'$.

Given a profinite group $G$ and a topological abelian group $V$ with a continuous action of $G$, we denote by $H^i (G, V)$ the continuous group cohomology, and by $h^i (G, V)$ its rank.

\section{The Chabauty-Kim method} \label{Chabauty-Kim}

In his papers \cite{kim05} and \cite{kim09}, Kim introduces a nonabelian analogue of the method of Chabauty-Coleman, often called \emph{Chabauty-Kim method} or \emph{nonabelian Chabauty}. We give here a quick introduction to it, roughly following the exposition of \cite{kim09}.

Let $X$ be a smooth curve over $\Q$, and let $X'$ be a smooth projective genus $g \geq 0$ curve over $\Q$ with $X \subset X'$ and $D = X' \setminus X$. Assume that $X$ is hyperbolic, that is $2g -2 + \# D(\alg{\Q}) > 0$. Also, assume that we are given smooth models $\mathcal{X}'$ of $X'$ and $\mathcal{D}$ of $D$ over $\Z_S$, for $S$ some finite set of primes, and let $\mathcal{X} = \mathcal{X}' \setminus \mathcal{D}$.

The Chabauty-Kim method studies the set of $S$-integral points $\mathcal{X}(\Z_S)$ of $\mathcal{X}$ through certain \emph{motivic unipotent Albanese maps}. Roughly speaking, the idea of these maps is to fix an $S$-integral base point $b$, choose a motivic unipotent fundamental group $U$ of $X$ with base point $b$, and map any other point $x$ to the class of the motivic $U$-torsor of paths from $b$ to $x$ in a suitable classifying space. In \cite{kim05} and \cite{kim09}, Kim considers the de Rham and the \'etale realisations of the fundamental group.

For the rest of this section, let us fix an $S$-integral base point $b$, a prime $p$ of good reduction for $\mathcal{X}'$, and an embedding of $\alg{\Q}$ inside $\alg{\Q}_p$. Also, let us set $T = S \cup \set{p}$.

Let $U^{\mathrm{dR}} = \pi_1^{\mathrm{dR}, \Q_p} (X_{\alg{\Q}_p}, b)$ be the $\Q_p$-pro-unipotent de Rham fundamental group of $X_{\alg{\Q}_p}$ with base point $b$. Let $U^{\mathrm{dR}, n}$ be the descending central series of $U^{\mathrm{dR}}$, and consider the finite dimensional quotients $U_n^{\mathrm{dR}} = U^{\mathrm{dR}, n + 1} \backslash U^{\mathrm{dR}}$ of $U^{\mathrm{dR}}$. Note that $U^{\mathrm{dR}}$ is naturally endowed with a decreasing Hodge filtration
\[ U^{\mathrm{dR}} \supset \cdots \supset F^i U^{\mathrm{dR}} \supset F^{i+1} U^{\mathrm{dR}} \supset \cdots \supset F^0 U^{\mathrm{dR}}, \]
which in turn induces decreasing filtrations on each $U_n^{\mathrm{dR}}$. The quotient $U^{\mathrm{dR}} / F^0 U^{\mathrm{dR}}$ naturally classifies de Rham path spaces, see \cite{kim09}, hence the \emph{de Rham unipotent Albanese map} can be defined to be the map
\[ j^{\mathrm{dR}} : \mathcal{X} (\Z_p) \ra U^{\mathrm{dR}} / F^0 U^{\mathrm{dR}} \]
which sends each $x$ to the class of the $U^{\mathrm{dR}}$-torsor of de Rham paths from $b$ to $x$. By passing to the quotients $U_n^{\mathrm{dR}}$, we get finite level versions
\[ j^{\mathrm{dR}}_n : \mathcal{X} (\Z_p) \ra U_n^{\mathrm{dR}} / F^0 U_n^{\mathrm{dR}}, \]
which fit into a natural compatible tower. Kim proves, see \cite[Theorem 1]{kim09}, that for each $n \geq 2$ the image of $j_n^{\mathrm{dR}}$ is Zariski dense. This property is of crucial importance for applications to Diophantine finiteness.

Let us now move to the \'etale side of the picture. Let $U^{\mathrm{\acute{e}t}} = \pi_1^{\mathrm{\acute{e}t}, \Q_p} (X_{\alg{\Q}}, b)$ be the $\Q_p$-pro-unipotent \'etale fundamental group of $X_{\alg{\Q}}$ with base point $b$, let $U^{\mathrm{\acute{e}t}, n}$ be the descending central series of $U^{\mathrm{\acute{e}t}}$, and let $U_n^{\mathrm{\acute{e}t}} = U^{\mathrm{\acute{e}t}, n + 1} \backslash U^{\mathrm{\acute{e}t}}$. In \cite{kim09}, Kim defines the \emph{Selmer variety} to be the provariety
\[ H^1_f (G_{\Q, T}, U^{\mathrm{\acute{e}t}}) \]
that classifies $G_{\Q, T}$-equivariant $U^{\mathrm{\acute{e}t}}$-torsors that are unramified outside $T$, and crystalline at $p$, and the \emph{\'etale unipotent Albanese map} to be the map
\[ j^{\mathrm{\acute{e}t}} : \mathcal{X} (\Z_S) \ra H^1_f (G_{\Q, T}, U^{\mathrm{\acute{e}t}}) \]
given by sending each $x$ to the class of the $G_{\Q, T}$-equivariant $U^{\mathrm{\acute{e}t}}$-torsor of \'etale paths from $b$ to $x$. Also in this case, we have finite level versions
\[ j^{\mathrm{\acute{e}t}}_n : \mathcal{X} (\Z_S) \ra H^1_f (G_{\Q, T}, U_n^{\mathrm{\acute{e}t}}). \]

Moreover, we also have local maps
\[ j^{\mathrm{\acute{e}t}}_{n, v} : \mathcal{X} (\Z_v) \ra H^1 (G_{\Q_v}, U_n^{\mathrm{\acute{e}t}}) \]
for each prime $v \neq p$, and
\[ j^{\mathrm{\acute{e}t}}_{n, p} : \mathcal{X} (\Z_p) \ra H^1_f (G_{\Q_p}, U_n^{\mathrm{\acute{e}t}}), \]
where $H^1_f (G_{\Q_p}, U_n^{\mathrm{\acute{e}t}})$ classifies $G_{\Q_p}$-equivariant $U_n^{\mathrm{\acute{e}t}}$-torsors that are crystalline.

A result of Kim and Tamagawa, see \cite[Corollary 0.2]{KT08}, gives that the maps $j^{\mathrm{\acute{e}t}}_{n, v}$, for $v \neq p$, have finite image. Following \cite{BDCKW18}, we define
\[ \mathcal{X} (\Z_p)_n = (j^{\mathrm{\acute{e}t}}_{n, p})^{-1} (\mathrm{loc}_p (\cap_{v \neq p} \mathrm{loc}_v^{-1}(\mathrm{im} \, j^{\mathrm{\acute{e}t}}_{n, v}))), \]
where $\mathrm{loc}_v$ are the naturally defined restriction maps. These sets clearly fit into a nested sequence
\[ \mathcal{X}(\Z_p) \supset \mathcal{X} (\Z_p)_1 \supset \mathcal{X} (\Z_p)_2 \supset \cdots \supset \mathcal{X} (\Z_p)_n \supset \cdots \supset \mathcal{X} (\Z_S). \]

Finally, Kim considers a nonabelian extension of Fontaine's Dieudonn\'e functor, which gives an isomorphism of varieties:
\[ D : H^1_f (G_{\Q_p}, U_n^{\mathrm{\acute{e}t}}) \isom U_{n}^{\mathrm{dR}} / F^0 U_{n}^{\mathrm{dR}}, \]
see \cite[Proposition 1.4]{kim12}. All these maps fit into the following fundamental commutative diagram.
\[ \xymatrix{
\mathcal{X}(\Z_S) \ar@{^{(}->}[r] \ar[d]^{j^{\mathrm{\acute{e}t}}_n} & \mathcal{X}(\Z_p) \ar[dr]^{j^{\mathrm{dR}}_n} \ar[d]^{j^{\mathrm{\acute{e}t}}_{n, p}} & \\
  H^1_f (G_{\Q, T}, U_n^{\mathrm{\acute{e}t}}) \ar[r]^{\mathrm{loc}_p} & H^1_f (G_{\Q_p}, U_n^{\mathrm{\acute{e}t}}) \ar[r]^{\sim}_{D} & U_{n}^{\mathrm{dR}} / F^0 U_{n}^{\mathrm{dR}}. 
  } \]
  
The key point is that the image of $\mathcal{X}(\Z_S)$ inside $U_n^{\mathrm{dR}} / F^0 U_n^{\mathrm{dR}} $ turns out to be contained inside the image of $H^1_f (G_{\Q, T}, U_n^{\mathrm{\acute{e}t}})$. We summarise the relation to Diophantine finiteness in the following lemma\footnote{This result is presented in alternative versions in the literature, from which this version can be easily deduced. We present its proof anyway for the sake of completeness.}.

\begin{lemma} \label{Diophantine_finiteness_lemma}
If for some $n \geq 2$ we have
\[ \dim H^1_f (G_{\Q, T}, U_n^{\mathrm{\acute{e}t}}) < \dim U_{n}^{\mathrm{dR}} / F^0 U_{n}^{\mathrm{dR}}, \]
then $\mathcal{X} (\Z_p)_{n}$ is finite.
\end{lemma}

\begin{proof}
We let $Y$ be the Zariski closure of the image of $H^1_f (G_{\Q, T}, U_n^{\mathrm{\acute{e}t}})$ inside $U_{n}^{\mathrm{dR}} / F^0 U_{n}^{\mathrm{dR}}$. Since $\dim H^1_f (G_{\Q, T}, U_n^{\mathrm{\acute{e}t}}) < \dim U_{n}^{\mathrm{dR}} / F^0 U_{n}^{\mathrm{dR}}$ , then $Y$ is a proper subset of $U_{n}^{\mathrm{dR}} / F^0 U_{n}^{\mathrm{dR}}$, and so there exists an algebraic function $\alpha \neq 0$ on $U_{n}^{\mathrm{dR}} / F^0 U_{n}^{\mathrm{dR}}$ which vanishes on $Y$. Since $j^{\mathrm{dR}}_n$ has Zariski dense image, we have that $\alpha \circ j^{\mathrm{dR}}_n \neq 0$ on $\mathcal{X} (\Z_p)$. Moreover, $\mathcal{X} (\Z_p)_n = (j^{\mathrm{dR}}_n)^{-1} (Y)$ is contained inside the zero locus of $\alpha \circ j^{\mathrm{dR}}_n$. This is finite by the $p$-adic Weierstrass preparation theorem.
\end{proof}

For the purposes of this paper, it is more convenient to replace, as in \cite{kim10} and \cite{CK10}, each motivic unipotent fundamental group $U$ with its maximal metabelian quotient, that is the quotient
\[ W = U / [U^2, U^2]. \]

Its finite dimensional quotients by the descending central series, as well as the corresponding Selmer varieties and unipotent Albanese maps are defined accordingly. Replacing $U$ with $W$ yields the same fundamental commutative diagram
\[ \xymatrix{
\mathcal{X}(\Z_S) \ar@{^{(}->}[r] \ar[d]^{j^{\mathrm{\acute{e}t}}_n} & \mathcal{X}(\Z_p) \ar[dr]^{j^{\mathrm{dR}}_n} \ar[d]^{j^{\mathrm{\acute{e}t}}_{n, p}} & \\
  H^1_f (G_{\Q, T}, W_n^{\mathrm{\acute{e}t}}) \ar[r]^{\mathrm{loc}_p} & H^1_f (G_{\Q_p}, W_n^{\mathrm{\acute{e}t}}) \ar[r]^{\sim}_{D} & W_{n}^{\mathrm{dR}} / F^0 W_{n}^{\mathrm{dR}}, 
  } \]
and Lemma \ref{Diophantine_finiteness_lemma} works verbatim.

Let us go back to our original setting, that is the genus $1$ hyperbolic curve $X$ obtained by removing the origin from an elliptic curve $E$ without complex multiplication. The main result of this paper is the following.

\begin{theorem} \label{main_theorem}
We have
\[ \dim H^1_f (G_{\Q, T}, W_n^{\mathrm{\acute{e}t}}) < \dim W_{n}^{\mathrm{dR}} / F^0 W_{n}^{\mathrm{dR}} \]
for $n$ sufficiently large.
\end{theorem}

The proof of this is postponed to \S \ref{proof_main_theorem}. Note that by Lemma \ref{Diophantine_finiteness_lemma} this immediately implies the required finiteness result.

\begin{corollary} \label{main_corollary}
The sets $\mathcal{X} (\Z_p)_n$ are finite for $n$ sufficiently large.
\end{corollary}

\section{Modularity and vanishing of Bloch-Kato Selmer groups} \label{Block-Kato_Selmer}

In this section, we prove the vanishing of certain Bloch-Kato Semer groups, which will turn out to be of crucial importance for our applications to the Chabauty-Kim method, using modularity techniques. Recall that for a number field $F$, a finite set $S$ of places of $F$ containing the infinite places, and a rational prime $p$ such that $S$ contains no places above $p$, given a representation $V$ of $G_{F}$ over $\Q_p$ unramified outside $S$, having set $T = S \cup \set{v \mid p}$, the Bloch-Kato Selmer group of $V$ is defined as
\[ H^1_f (G_{F, T}, V) = \ker \bigl( H^1 (G_{F, T}, V) \ra \textstyle{\prod_{v \mid p}} H^1 (G_{F_v}, B_{\mathrm{cris}} \otimes V) \bigr), \]
where $B_{\mathrm{cris}}$ denotes Fontaine's crystalline ring.

Moving back to our original setting, we denote by $\rho_{E, p} : G_{\Q} \ra \GL_2 (\Z_p)$ the representation of $G_{\Q}$ on the $p$-adic Tate module $T_p (E)$ of $E$, and we let $V_p (E) = T_p (E) \otimes \Q_p$. By a classical result of Serre, see \cite[\S 4]{serre72}, for every finite extension $F$ of $\Q$ we have that for $p$ large enough the representation $\rho_{E, p} \! \mid_{G_F}$ is surjective. For the rest of this paper, given any finite extension $F$ of $\Q$, up to eventually choosing a bigger $p$, we assume that this condition is satisfied. As above, we let $T = S \cup \set{p}$. For any algebraic extension $F$ of $\Q$, we still denote by $T$, with a slight abuse of notation, the set of places of $F$ above those in $T$. We prove the following result\footnote{This result improves a previous result by Allen, see \cite[Theorem 3.3.1]{allen16}, using the potential automorphy results of \cite{allen_et_al18} and the vanishing results for adjoint Bloch-Kato Selmer groups of \cite{NT19}.}.

\begin{theorem} \label{vanishing_Bloch-Kato_Selmer}
Let $F$ be an imaginary quadratic field. Then we have
\[ H^1_f (G_{\Q, T}, \mathrm{Sym}^{2n} V_p(E) \otimes \epsilon^{-n} \delta_{F/\Q}^{n+1} ) = 0 \]
for all $n \geq 1$.
\end{theorem}

\begin{proof}
For simplicity of notation, let us write $V = V_p (E)$. The Weil pairing induces an isomorphism $V^\vee \cong V \otimes \epsilon^{-1}$, and so we have isomorphisms $(\mathrm{Sym}^n V)^\vee \cong \mathrm{Sym}^n V \otimes \epsilon^{-n}$ for every $n \geq 1$. It follows that the representation of $G_{\Q}$ on $\mathrm{Sym}^n V$ factors through $\mathrm{GO}_{n+1} (\Q_p)$ with totally even multiplier when $n$ is even, and through $\mathrm{GSp}_{n+1} (\Q_p)$ with totally odd multiplier when $n$ is odd.

For every $n \geq 1$, the natural symmetric pairing $\mathrm{Sym}^n V \times \mathrm{Sym}^n V \ra \mathrm{Sym}^{2n} V$ induces a surjection $\mathrm{Sym}^2 (\mathrm{Sym}^n V) \ra \mathrm{Sym}^{2n} V$, which in turn induces a surjection
\[ \mathrm{Sym}^{2} (\mathrm{Sym}^n V) \otimes \epsilon^{-n} \delta_{F/\Q}^{n+1} \ra \mathrm{Sym}^{2n} V \otimes \epsilon^{-n} \delta_{F/\Q}^{n+1} \]
which is $G_{\Q, T}$-equivariant, and has a $G_{\Q, T}$-equivariant splitting. It follows that it is enough to show that $H^1_f (G_{\Q, T}, \mathrm{Sym}^{2} (\mathrm{Sym}^n V) \otimes \epsilon^{-n} \delta_{F/\Q}^{n+1}) = 0$ for all $n \geq 1$.

Note that we have a decomposition as $G_{\Q}$-representations
\[ \mathrm{ad} (\mathrm{Sym}^n V \! \mid_{G_{F}}) \cong \mathrm{Sym}^2 (\mathrm{Sym}^n V) \otimes \epsilon^{-n} \delta_{F/\Q} \oplus \wedge^2 (\mathrm{Sym}^n V) \otimes \epsilon^{-n} \]
if $n$ is even, and
\[ \mathrm{ad} (\mathrm{Sym}^n V \! \mid_{G_{F}}) \cong \mathrm{Sym}^2 (\mathrm{Sym}^n V) \otimes \epsilon^{-n} \oplus \wedge^2 (\mathrm{Sym}^n V) \otimes \epsilon^{-n}  \delta_{F/\Q} \]
if $n$ is odd. In both cases, we have that $\mathrm{Sym}^{2} (\mathrm{Sym}^n V) \otimes \epsilon^{-n} \delta_{F/\Q}^{n+1}$ as a $G_{\Q}$-representation is a direct summand of $\mathrm{ad} (\mathrm{Sym}^n V \! \mid_{G_{F}})$. It therefore follows that it is enough to prove that $H^1_f (G_{\Q, T}, \mathrm{ad} (\mathrm{Sym}^n V \! \mid_{G_{F}})) = 0$ for all $n \geq 1$.

By \cite[Theorem 7.1.9]{allen_et_al18} we have that for every $n \geq 1$ there exists a finite CM Galois extension $F'$ of $F$ such that the representation $\mathrm{Sym}^n \, V \! \mid_{G_{F'}}$ is automorphic. Let us fix $n \geq 1$ and a CM field $F'$ as above, and let us denote by $F'^+$ the maximal totally real subfield of $F'$. By considering restriction-corestriction and inflation, we get an injection $H^1 (G_{\Q, T}, \mathrm{ad} (\mathrm{Sym}^n V \! \mid_{G_{F}})) \ra H^1 (G_{F'^+, T}, \mathrm{ad} (\mathrm{Sym}^n V \! \mid_{G_{F'}}))$ such that the following diagram commutes
\[ \xymatrix{
 H^1 (G_{\Q, T}, \mathrm{ad} (\mathrm{Sym}^n V \! \mid_{G_{F}})) \ar[d] \ar[r] & H^1 (G_{\Q_p}, B_{\mathrm{cris}} \otimes \mathrm{ad} (\mathrm{Sym}^n V \! \mid_{G_{F}})) \ar[d] \\
H^1 (G_{F'^+, T}, \mathrm{ad} (\mathrm{Sym}^n V \! \mid_{G_{F'}})) \ar[r] & \prod_{w \mid p} H^1 (G_{F'^+_w}, B_{\mathrm{cris}} \otimes \mathrm{ad} (\mathrm{Sym}^n V \! \mid_{G_{F'}})).} \]

We then get an injection $H^1_f (G_{\Q, T}, \mathrm{ad} (\mathrm{Sym}^n V \! \mid_{G_{F}})) \ra H^1_f (G_{F'^+, T}, \mathrm{ad} (\mathrm{Sym}^n V \! \mid_{G_{F'}}))$, and so it is enough to prove that $H^1_f (G_{F', T}, \mathrm{ad} (\mathrm{Sym}^n V \! \mid_{G_{F'}})) = 0$. In particular, up to eventually choosing a bigger $p$, we can assume that each $\mathrm{Sym}^n \, V \! \mid_{G_F}$ is automorphic.

Since $\rho_{E, p} \! \mid_{G_F}$ is surjective, we have that $\mathrm{Sym}^n \, \rho_{E, p} (G_{F(\zeta_{p^\infty})})$ is an enormous subgroup of $\GL_{n+1} (\Z_p)$ for every $n \geq 1$, see \cite[Example 2.34]{NT19}. Then, since $\mathrm{Sym}^n \, V \! \mid_{G_F}$ is automorphic, by \cite[Theorem 5.2]{NT19} we get that $H^1_f (G_{\Q, T}, \ad (\mathrm{Sym}^n V \! \mid_{G_F})) = 0$ for every $n \geq 1$. The conclusion follows.
\end{proof}

\begin{remark} \label{totally_real_CM}
Note that our proof works verbatim when $\Q$ is replaced by any totally real number field.
\end{remark}

\section{Iwasawa theory and vanishing of Selmer groups} \label{Selmer}

In this section, we prove a vanishing result for Selmer groups, using some Iwasawa theory and the vanishing results for Bloch-Kato Selmer groups proved in the previous section. Let us first of all introduce some notation. Given a representation $V$ of $G_{F}$ over $\Q_p$ unramified outside a finite set $S$ of places of $F$ containing the infinite places and no places above $p$, having set $T = S \cup \set{v \mid p}$, we recall that the $i$-th Selmer group of $V$ is defined as
\[ \Sha^i (G_{F, T}, V) = \ker \bigl( H^i (G_{F, T}, V) \ra \textstyle{\prod_{v \in T}} H^i (G_{F_v}, V) \bigr). \]

For $i = 1$, it is immediate to notice that we have an inclusion $\Sha^1 (G_{F, T}, V) \subset H^1_f (G_{F, T}, V)$. Moreover, we have

\begin{lemma} \label{lemma_equivalence}
Let $V^D = V^\vee \otimes \epsilon$ denote the Cartier dual of $V$. Then:
\begin{enumerate}
\item[$(1)$] If $H^2 (G_{F, T}, V^D) = 0$, then $\Sha^1 (G_{F, T}, V) = 0$.
\item[$(2)$] If $V$ is geometric and pure of nonzero weight\footnote{Recall that a geometric representation $V$ of $G_{F}$ over $\Q_p$ is said to be pure of weight $w \in \R$ if for each finite place $v$ of $F$ the Weil-Deligne representation attached to $V \! \mid_{G_{F_v}}$ is pure of weight $w$ in the sense of \cite[\S 1]{TY07}.}, and if $\Sha^1 (G_{F, T}, V) = 0$, then $H^2 (G_{F, T}, V^D) = 0$.
\end{enumerate}
\end{lemma}

\begin{proof}
By Poitou-Tate duality, we have an exact sequence
\[ H^2 (G_{F, T}, V^D)^\vee \ra H^1 (G_{F, T}, V) \ra \textstyle{\prod_{v \in T}} H^1 (G_{F_v}, V), \]
and so part $(1)$ immediately follows.

Moving to part $(2)$, again by Poitou-Tate duality, we have that $\Sha^2 (G_{F, T}, V^D) \cong \Sha^1 (G_{F, T}, V)^\vee$, and so that $\Sha^2 (G_{F, T}, V^D) = 0$. It follows that the map
\[ H^2 (G_{F, T}, V^D) \ra \textstyle{\prod_{v \in T}} H^2 (G_{F_v}, V^D) \]
is injective. Now, by local Tate duality, for every $v \in T$ we have an isomorphism $H^2 (G_{F_v}, V^D) \cong H^0 (G_{F_v}, V)^\vee$. In order to get the conclusion, it is then enough to show that $H^0 (G_{F_v}, V) = 0$ for each $v \in T$.

Let $v \in T$. If $v \nmid p$, since the Weil-Deligne representation attached to $V \! \mid_{G_{F_v}}$ is pure of nonzero weight by assumption, we have that $H^0 (G_{F_v}, V) = \Hom_{G_{F_v}} (\Q_p, V) = 0$. If $v \mid p$, since $V \! \mid_{G_{F_v}}$ is unramified, then it is crystalline, and we have $H^0(G_{F_v}, V) = \Hom_{G_{F_v}} (\Q_p, V) \cong \Hom_{\mathrm{MF}(\phi)} (F_v^{\mathrm{nr}}, D_{\mathrm{cris}} (V))$, where $\mathrm{MF}(\phi)$ denotes Fontaine's category of admissible filtered $\phi$-modules, $F_v^{\mathrm{nr}}$ denotes the maximal absolutely unramified subfield of $F_v$, and $D_{\mathrm{cris}}$ denotes Fontaine's crystalline functor. Since the Weil-Deligne representation attached to $V \! \mid_{G_{F_v}}$ is pure of nonzero weight by assumption, we get that $\Hom_{\mathrm{MF}(\phi)} (F_v^{\mathrm{nr}}, D_{\mathrm{cris}} (V)) = 0$, and so the conclusion follows.
\end{proof}

Back to our original setting, we let $\Q^\infty$ be the field generated by the $p$-power torsion points of $E$ over $\Q$. Note that the Weil pairing gives an inclusion $\Q(\zeta_{p^\infty}) \subset \Q^\infty$. Also, for our choice of $p$ the representation $\rho_{E, p}$ induces an isomorphism $\Gal (\Q^\infty / \Q) \cong \GL_2 (\Z_p)$. Write $\Gamma = \Gal (\Q^\infty / \Q)$, and let $\Lambda = \Z_p [ \! [ \Gamma ] \! ]$ be the Iwasawa algebra of $\Gamma$. Let $H^{\infty}_T$ be the maximal unramified abelian pro-$p$ extension of $\Q^\infty$ in which all primes above $T$ split completely, and let $\mathcal{X}_T^\infty = \Gal (H^\infty_T / \Q^\infty)$, which is naturally a $\Lambda$-module. We prove the following lemma.

\begin{lemma} \label{Selmer_Iwasawa_lemma}
Let $V$ be a semisimple representation of $\GL(V_p (E))$ over $\Q_p$ which does not contain $\Q_p$ as an irreducible subrepresentation. Then, considering $V$ as a representation of $G_{\Q}$ via $\rho_{E, p}$, we have an isomorphism
\[\Sha^1 (G_{\Q, T}, V) \cong \Hom_{\Lambda} (\mathcal{X}_T^\infty, V). \]
\end{lemma}

\begin{proof}
Since $\Gamma$ is isomorphic to $\GL_2 (\Z_p)$ via $\rho_{E, p}$, we have an isomorphism $H^* (\Gamma, V) \cong H^* (\GL_2 (\Z_p), V)$. Now, since $V$ does not contain $\Q_p$ as an irreducible subrepresentation, by Lazard's isomorphism, see for instance \cite[Theorem 5.2.4]{SW00}, and the fact that for the Lie algebra cohomology we have $H^i (\gl_2 (\Z_p), V) = 0$ for $i \geq 0$ by \cite[Theorem 10]{HS53}, we get that $H^i (\GL_2 (\Z_p), V) = 0$ for $i \geq 0$. It follows that $H^i (\Gamma, V) = 0$ for $i \geq 0$.

Since the induced action of $G_{\Q^\infty, T}$ on $V$ is trivial, the Hochschild-Serre spectral sequence for the closed normal subgroup $G_{\Q^\infty, T}$ of $G_{\Q, T}$ gives an exact sequence
\[ H^1 (\Gamma, V) \ra H^1 (G_{\Q, T}, V) \ra \Hom (G_{\Q^\infty, T}, V)^{\Gamma} \ra H^2 (\Gamma, V). \]

We then get an isomorphism $H^1 (G_{\Q, T}, V) \cong \Hom (G_{\Q^\infty, T}, V)^{\Gamma} = \Hom_\Lambda (G_{\Q^\infty, T}, V)$, which in turn induces an isomorphism $\Sha^1 (G_{\Q, T}, V) \cong \Hom_\Lambda (\mathcal{X}_T^\infty, V)$.
\end{proof}

We prove the following result.

\begin{theorem} \label{vanishing_Selmer}
We have
\[ \Sha^1 (G_{\Q, T}, \mathrm{Sym}^{n} V_p (E) \otimes \epsilon^{-n}) = 0 \]
for all $n \geq 2$.
\end{theorem}

\begin{proof}
By part $(1)$ of Lemma \ref{lemma_equivalence} it is enough to prove that $H^2 (G_{\Q, T}, \mathrm{Sym}^{n} V_p(E) \otimes \epsilon) = 0$. Let us write $n = 2k + h$, where $k \geq 1$ is odd, and $h \in \set{0, 1, 2, 3}$. The natural surjection $\mathrm{Sym}^{2k} V_p(E) \otimes \mathrm{Sym}^{h} V_p(E) \ra \mathrm{Sym}^{n} V_p(E)$ induces a surjection
\[ \mathrm{Sym}^{2k} V_p(E) \otimes \mathrm{Sym}^{h} V_p(E) \otimes \epsilon \ra \mathrm{Sym}^{n} V_p(E) \otimes \epsilon \]
which is $G_{\Q, T}$-equivariant. Since $G_{\Q, T}$ has $p$-cohomological dimension $2$, we get a surjection
\[ H^2 (G_{\Q, T}, \mathrm{Sym}^{2k} V_p(E) \otimes \mathrm{Sym}^h V_p(E) \otimes \epsilon) \ra H^2 (G_{\Q, T}, \mathrm{Sym}^{n} V_p(E) \otimes \epsilon), \]
and so it is enough to prove that $H^2 (G_{\Q, T}, \mathrm{Sym}^{2k} V_p(E) \otimes \mathrm{Sym}^h V_p(E) \otimes \epsilon) = 0$. Moreover, since $\mathrm{Sym}^{2k} V_p(E) \otimes \mathrm{Sym}^h V_p(E) \otimes \epsilon^{-n}$ is pure of weight $n \neq 0$, by part $(2)$ of Lemma \ref{lemma_equivalence} it is enough to prove that $\Sha^1 (G_{\Q, T}, \mathrm{Sym}^{2k} V_p (E) \otimes \mathrm{Sym}^h V_p(E) \otimes \epsilon^{-n}) = 0$.

Let $V = \mathrm{Sym}^{2k} V_p (E) \otimes \mathrm{Sym}^h V_p(E) \otimes \epsilon^{-n}$ and $V' = \mathrm{Sym}^{2k} V_p (E) \otimes \epsilon^{-k}$, and let $\mathcal{V} = \mathrm{Sym}^{2k} T_p (E) \otimes \mathrm{Sym}^h T_p(E) \otimes \epsilon^{-n}$ and $\mathcal{V}' = \mathrm{Sym}^{2k} T_p (E) \otimes \epsilon^{-k}$ be $G_{\Q, T}$-stable $\Z_p$-lattices, so that $\mathcal{V} \cong \mathcal{V}' \otimes \mathrm{Sym}^h T_p(E) \otimes \epsilon^{-n+k}$. Since $V$ is pure of weight $n \neq 0$, and so it cannot contain $\Q_p$ as a $\GL(V_p(E))$-subrepresentation, and $V'$ is nontrivial and irreducible as a $\GL(V_p(E))$-representation, by Lemma \ref{Selmer_Iwasawa_lemma} we have that $\Sha^1 (G_{\Q, T}, V) \cong \Hom_{\Lambda} (\mathcal{X}_T^\infty, V)$ and $\Sha^1 (G_{\Q, T}, V') \cong \Hom_{\Lambda} (\mathcal{X}_T^\infty, V')$. Since $\Sha^1 (G_{\Q, T}, V') \subset H^1_f (G_{\Q, T}, V')$, by Theorem \ref{vanishing_Bloch-Kato_Selmer} we have that $\Sha^1 (G_{\Q, T}, V') = 0$, and hence that $\Hom_{\Lambda} (\mathcal{X}_T^\infty, V') = 0$.

On the other hand, since $\Q^\infty$ contains the cyclotomic $\Z_p$-extension of $\Q$, we have that $\mathcal{X}_T^\infty$ is a finitely generated torsion $\Lambda$-module by \cite[Lemma 3.4]{HS05}, and hence finitely presented, as $\Lambda$ is Noetherian. Moreover, $\mathrm{Sym}^h T_p(E) \otimes \epsilon^{-n+k}$ is $\Z_p$-flat, and so \cite[Proposition I.9.10]{bourbaki_commutative_algebra} gives us an isomorphism of $\Z$-modules $\Hom_{\Lambda} (\mathcal{X}_T^\infty, \mathcal{V}) \cong \Hom_{\Lambda} (\mathcal{X}_T^\infty, \mathcal{V}') \otimes \mathrm{Sym}^h T_p(E) \otimes \epsilon^{-n+k}$, which in turn induces an isomorphism $\Hom_{\Lambda} (\mathcal{X}_T^\infty, V) \cong \Hom_{\Lambda} (\mathcal{X}_T^\infty, V') \otimes \mathrm{Sym}^h V_p(E) \otimes \epsilon^{-n+k}$. It follows that $\Hom_{\Lambda} (\mathcal{X}_T^\infty, V) = 0$, hence that $\Sha^1 (G_{\Q, T}, V) = 0$.
\end{proof}

\begin{remark}
We remark that also the results of this section can be immediately extended to any totally real base field.
\end{remark}

\section{Proof of Theorem \ref{main_theorem}} \label{proof_main_theorem}

Let us start by proving the following lemma.

\begin{lemma} \label{lemma_quotient}
We have a $G_{\Q, T}$-equivariant isomorphism
\[ W^{\mathrm{\acute{e}t}, n+1} \backslash W^{\mathrm{\acute{e}t}, n} \cong \mathrm{Sym}^{n-2} V_p(E) \otimes \epsilon \]
for every $n \geq 2$.
\end{lemma}

\begin{proof}
Let us first of all prove that $W^{\mathrm{\acute{e}t}, n+1} \backslash W^{\mathrm{\acute{e}t}, n}$ is a quotient of $\mathrm{Sym}^{n-2} V_p(E) \otimes \epsilon$ as a $G_{\Q, T}$-module\footnote{The proof of this fact has been suggested to us by Minhyong Kim, which we thank once more.}.

Let $\mathfrak{u}^{\mathrm{\acute{e}t}}$ be the Lie algebra of $U^{\mathrm{\acute{e}t}}$, and let $\mathfrak{u}^{\mathrm{\acute{e}t}, n}$ be its descending central series. We have a natural $G_{\Q, T}$-equivariant isomorphism
\[ U^{\mathrm{\acute{e}t}, n+1} \backslash U^{\mathrm{\acute{e}t}, n} \cong \mathfrak{u}^{\mathrm{\acute{e}t}, n+1} \backslash \mathfrak{u}^{\mathrm{\acute{e}t}, n}. \]

Let $\mathfrak{u}_1^{\mathrm{\acute{e}t}} = \mathfrak{u}^{\mathrm{\acute{e}t}, 2} \backslash \mathfrak{u}^{\mathrm{\acute{e}t}}$, so that $\mathfrak{u}_1^{\mathrm{\acute{e}t}} \cong V_p (E)$. By induction on $n \geq 2$, using the identity $\mathfrak{u}^{\mathrm{\acute{e}t}, n+1} \backslash \mathfrak{u}^{\mathrm{\acute{e}t}, n} = [\mathfrak{u}_1^{\mathrm{\acute{e}t}}, \mathfrak{u}^{\mathrm{\acute{e}t}, n} \backslash \mathfrak{u}^{\mathrm{\acute{e}t}, n-1}]$, we see that $\mathfrak{u}^{\mathrm{\acute{e}t}, n+1} \backslash \mathfrak{u}^{\mathrm{\acute{e}t}, n}$ is generated as a $G_{\Q, T}$-module by elements of the form $[a_1, [a_2, [ \cdots [a_{n-2}, b] \cdots ]]]$, where $a_i \in \mathfrak{u}_1^{\mathrm{\acute{e}t}}$, and $b \in \mathfrak{u}^{\mathrm{\acute{e}t}, 3} \backslash \mathfrak{u}^{\mathrm{\acute{e}t}, 2}$. Note that as a $G_{\Q, T}$-module $\mathfrak{u}^{\mathrm{\acute{e}t}, 3} \backslash \mathfrak{u}^{\mathrm{\acute{e}t}, 2}$ is a quotient of $\wedge^2 \mathfrak{u}_1^{\mathrm{\acute{e}t}}$.

Let now $\mathfrak{w}^{\mathrm{\acute{e}t}}$ be the Lie algebra of $W^{\mathrm{\acute{e}t}}$, and let $\mathfrak{w}^{\mathrm{\acute{e}t}, n}$ be its descending central series. For every generator $[a_1, [a_2, [ \cdots [a_{n-2}, b] \cdots ]]]$ of $\mathfrak{u}^{\mathrm{\acute{e}t}, n+1} \backslash \mathfrak{u}^{\mathrm{\acute{e}t}, n}$ as above, we have that modulo $[\mathfrak{u}^{\mathrm{\acute{e}t}, 2}, \mathfrak{u}^{\mathrm{\acute{e}t}, 2}]$ the order of the $a_i$'s can be normalised to any fixed order. It follows that as a $G_{\Q, T}$-module
\[ W^{\mathrm{\acute{e}t}, n+1} \backslash W^{\mathrm{\acute{e}t}, n} \cong \mathfrak{w}^{\mathrm{\acute{e}t}, n+1} \backslash \mathfrak{w}^{\mathrm{\acute{e}t}, n} \]
is a quotient of
\[ \mathrm{Sym}^{n - 2} \mathfrak{u}_1^{\mathrm{\acute{e}t}} \otimes \wedge^2 \mathfrak{u}_1^{\mathrm{\acute{e}t}} \cong \mathrm{Sym}^{n-2} V_p(E) \otimes \epsilon. \]

Since $\mathrm{Sym}^{n-2} V_p(E) \otimes \epsilon$ is irreducible, and $W^{\mathrm{\acute{e}t}, n+1} \backslash W^{\mathrm{\acute{e}t}, n} \neq 0$, we conclude that as $G_{\Q, T}$-modules $W^{\mathrm{\acute{e}t}, n+1} \backslash W^{\mathrm{\acute{e}t}, n} \cong \mathrm{Sym}^{n-2} V_p(E) \otimes \epsilon$.
\end{proof}

Combining this with the results of the previous section we immediately get

\begin{corollary} \label{vanishing_H2}
We have
\[ H^2 (G_{\Q, T}, W^{\mathrm{\acute{e}t}, n+1} \backslash W^{\mathrm{\acute{e}t}, n}) = 0 \]
for all $n \geq 4$.
\end{corollary}

We can finally prove our main result.

\begin{proof} [Proof of Theorem \ref{main_theorem}]
Assume that $n \geq 2$. For both the de Rham and the \'etale realisations, let us consider the exact sequence
\[ 0 \ra W^{n+1} \backslash W^n \ra W_{n} \ra W_{n-1} \ra 0. \]

In the de Rham realisation, we get that
\begin{align*}
 \dim (W_{n}^{\mathrm{dR}} / F^0 W_{n}^{\mathrm{dR}}) &- \dim (W_{n-1}^{\mathrm{dR}} / F^0 W_{n-1}^{\mathrm{dR}}) \\
 &= \dim (W^{\mathrm{dR}, n+1} \backslash W^{\mathrm{dR}, n}) - \dim F^0(W^{\mathrm{dR}, n+1} \backslash W^{\mathrm{dR}, n})
 \end{align*}
for $n \geq 2$. Following \cite[\S 3 and \S 4]{kim09} and \cite[\S 3]{kim10}, we get that
\[ \dim (W_{2}^{\mathrm{dR}} / F^0 W_{2}^{\mathrm{dR}}) = 2 \]
and that
\[ \dim F^0(W^{\mathrm{dR}, n+1} \backslash W^{\mathrm{dR}, n}) \leq 1 \]
for $n \geq 3$, so that
\[ \dim (W_{n}^{\mathrm{dR}} / F^0 W_{n}^{\mathrm{dR}}) \geq 3 + \frac{n(n-3)}{2} \]
for $n \geq 3$.

Let us now move to the \'etale side of the picture. First of all, we have that
\[ \dim H^1_f (G_{\Q, T}, W_{n}^{\mathrm{\acute{e}t}}) \! - \! \dim H^1_f (G_{\Q, T}, W_{n-1}^{\mathrm{\acute{e}t}}) \leq \dim H^1_f (G_{\Q, T}, W^{\mathrm{\acute{e}t}, n+1} \backslash W^{\mathrm{\acute{e}t}, n}) \]
for $n \geq 2$. Let now $s$ be the cardinality of $S$, and let $r = \dim H^1_f (G_{\Q, T}, V_p (E))$. With these notations, following \cite[\S 3]{kim10}, we get that
\[ \dim H^1_f (G_{\Q, T}, W_{2}^{\mathrm{\acute{e}t}}) \leq r + s - 1, \]
and hence that
\[ \dim H^1_f (G_{\Q, T}, W_{n}^{\mathrm{\acute{e}t}}) \leq r + s - 1 + \sum_{i = 3}^n h^1 (G_{\Q, T}, W^{\mathrm{\acute{e}t}, i+1} \backslash W^{\mathrm{\acute{e}t}, i}) \]
for $n \geq 3$. By Lemma \ref{lemma_quotient} we have a $G_{\Q, T}$-equivariant isomorphism $W^{\mathrm{\acute{e}t}, i+1} \backslash W^{\mathrm{\acute{e}t}, i} \cong \mathrm{Sym}^{i-2} V_p(E) \otimes \epsilon$ for every $i \geq 3$. Let $r' = h^1 (G_{\Q, T}, W^{\mathrm{\acute{e}t}, 4} \backslash W^{\mathrm{\acute{e}t}, 3}) = h^1 (G_{\Q, T}, V_p (E) \otimes \epsilon)$. For $i \geq 4$ the Euler characteristic formula gives that
\[ h^1 (G_{\Q, T}, W^{\mathrm{\acute{e}t}, i+1} \backslash W^{\mathrm{\acute{e}t}, i}) = h^2 (G_{\Q, T}, W^{\mathrm{\acute{e}t}, i+1} \backslash W^{\mathrm{\acute{e}t}, i}) + \dim (W^{\mathrm{\acute{e}t}, i+1} \backslash W^{\mathrm{\acute{e}t}, i})^{c=-1}, \]
where $c$ denotes the complex conjugation induced by $\infty$. By Corollary \ref{vanishing_H2} we have that
\[h^2 (G_{\Q, T}, W^{\mathrm{\acute{e}t}, i+1} \backslash W^{\mathrm{\acute{e}t}, i}) = 0. \]

Also, since $\mathrm{Sym}^{i-2} V_p(E)$ becomes automorphic over a finite totally real Galois extension of $\Q$ by \cite[Theorem 7.1.4]{BLGG11}, so does $\mathrm{Sym}^{i-2} V_p(E) \otimes \epsilon$, and so we have that
\[ \dim (W^{\mathrm{\acute{e}t}, i+1} \backslash W^{\mathrm{\acute{e}t}, i})^{c=-1} =  \left \lfloor \frac{i-1}{2} \right \rfloor \text{ or } \left \lceil \frac{i-1}{2} \right \rceil  \]
by \cite[Theorem 1.1]{CLH16}. Moreover, since $\mathrm{Sym}^{i-2} V_p(E) \otimes \epsilon$ has determinant $\epsilon^{i-1}$, and $\epsilon^{i-1} (c) = (-1)^{i-1}$, we deduce that
\[ \dim (W^{\mathrm{\acute{e}t}, i+1} \backslash W^{\mathrm{\acute{e}t}, i})^{c=-1} =  \left \lfloor \frac{i-1}{2} \right \rfloor + \begin{cases}
1 & \text{if } i \equiv 2 \pmod 4, \\
0 & \text{otherwise.}
\end{cases} \]

If $n$ is divisible by $4$, we then get that\footnote{We assume that $n$ is divisible by $4$ just to simplify the computations. Note that this assumption affects our final lower bound estimate only of an error term $\leq 3$.}
\[ \dim H^1_f (G_{\Q, T}, W_{n}^{\mathrm{\acute{e}t}}) \leq r + r' + s + n \left( \frac{n}{4} + 1 \right) + \frac{3n}{4}. \]

In conclusion, if we choose $n_0$ to be the smallest integer $n \geq 4$ divisible by $4$ and such that
\[ r + r' + s + n \left( \frac{n}{4} + 1 \right) + \frac{3n}{4} < 3 + \frac{n(n-3)}{2}, \]
we get that
\[ \dim H^1_f (G_{\Q, T}, W_n^{\mathrm{\acute{e}t}}) < \dim W_{n}^{\mathrm{dR}} / F^0 W_{n}^{\mathrm{dR}} \]
for every $n \geq n_0$ as required.
\end{proof}

\begin{remark}
Keep notations as in the proof of Theorem \ref{main_theorem}. We remark that the Block-Kato conjectures, see for instance \cite[\S 4.2.2]{FPR94}, predict $r$ to be equal to the analytic rank of $E$ (nevertheless, using the Kummer map, we always have that $r$ is greater than or equal to the algebraic rank of $E$), and a conjecture of Jannsen, see \cite[\S 6]{jannsen89}, predicts that $r' = 1$.
\end{remark}

\bibliographystyle{abbrv}
\bibliography{Bibliography.bib}

\end{document}